\documentclass[12pt,final]{amsart}

\usepackage{amsmath, amsxtra, amssymb, mathdots}
\usepackage{appendix}
\usepackage{verbatim}
\usepackage{multirow}

\usepackage{graphicx}

\usepackage[cmtip,all]{xy}

\usepackage[notcite, notref]{showkeys}

\usepackage[colorlinks, linkcolor=blue, citecolor=red, urlcolor=black]{hyperref}
\usepackage[margin=1.5in]{geometry}
\newtheorem{theorem}{Theorem}[section]
\newtheorem{lemma}[theorem]{Lemma}
\newtheorem{prop}[theorem]{Proposition}
\newtheorem{corollary}[theorem]{Corollary}
\newtheorem{claim}[theorem]{Claim}

\theoremstyle{definition}
\newtheorem{definition}[theorem]{Definition}
\newtheorem{example}[theorem]{Example}

\theoremstyle{remark}
\newtheorem{remark}[theorem]{Remark}

\numberwithin{equation}{section}

\DeclareMathOperator{\Sym}{Sym}
\DeclareMathOperator{\proj}{Proj}

\DeclareMathOperator{\Jac}{Jac}

\DeclareMathOperator{\chark}{char}

\DeclareMathOperator{\SL}{SL}
\DeclareMathOperator{\GL}{GL}

\DeclareMathOperator{\codim}{codim}

\DeclareMathOperator{\Grass}{Grass}

\DeclareMathOperator{\Res}{Res}
\DeclareMathOperator{\Hom}{Hom}

\def\AA{\mathbf{A}}

\newcommand{\gitq}{/\hspace{-0.25pc}/}

\def\A{\mathcal{A}}
\def\co{\colon\thinspace} 
\def\mi{\mathfrak{m}}

\def\cA{\mathcal{A}}

\def\RR{\mathbb{R}}
\def\PP{\mathbb{P}}

\def\CC{\mathbb{C}}
\def\HH{\mathrm{H}}
\def\cO{\mathcal{O}}

\DeclareMathOperator{\Soc}{Soc}

\begin{document}

\title[Stability of associated forms]{Stability of associated forms}

\author[Maksym Fedorchuk]{Maksym Fedorchuk}
\address[Fedorchuk]{Department of Mathematics\\
Boston College\\
140 Commonwealth Ave\\
Chestnut Hill, MA 02467, USA}
\email{maksym.fedorchuk@bc.edu}

\author[Alexander Isaev]{Alexander Isaev}
\address[Isaev]{Mathematical Sciences Institute\\
Australian National University\\
Acton, Canberra, ACT 2601, Australia}
\email{alexander.isaev@anu.edu.au}

\begin{abstract}We show that the associated form, or, equivalently, a Macaulay inverse system, 
of an Artinian complete intersection of type $(d,\dots, d)$ is polystable. As an application, we obtain an invariant-theoretic variant of the Mather-Yau theorem for homogeneous hypersurface singularities.
\end{abstract}

\maketitle

\section{Introduction}
In this paper, we establish the GIT polystability of 
Macaulay inverse systems for Gorenstein Artin algebras
given by balanced complete intersections.  
This leads to a purely invariant-theoretic solution to the problem of deciding when two such algebras
are isomorphic. An important example of a balanced complete intersection is the 
Milnor algebra of an isolated homogeneous hypersurface singularity and so, as an application
of our polystability result, we obtain an algebraic variant of the Mather-Yau theorem for such singularities
over an arbitrary field of characteristic zero. 

We will now explain our approach. Recall that 
for two homogeneous forms $f(x_1,\dots,x_n)$ and $g(x_1,\dots,x_n)$ over a field $k$, 
the problem of determining whether one can be obtained from the other by a linear
change of variables can often be solved 
by a purely algebraic method offered by Geometric Invariant Theory (GIT).  Namely, if
$\deg f=\deg g=d$, then the question can be rephrased as whether we have the equality of the orbits
\[
\GL(n)\cdot f=\GL(n)\cdot g
\]
under the natural action of $\GL(n)$ on  
$\Sym^d V$, where $V$ is the standard representation of $\GL(n)$. 
When $f$ and $g$ are polystable in the sense of GIT, 
their orbits in $\Sym^d V$ can be distinguished using invariants. Namely, the two orbits 
are distinct if and only if there exists a homogeneous $\SL(n)$-invariant $\mathfrak{I}$ on $\Sym^d V$,
that is, an $\SL(n)$-invariant homogeneous element of $\Sym \bigl(\Sym^d V\bigr)^{\vee}$, such that  
$\mathfrak{I}(f)= 0$ and $\mathfrak{I}(g)\neq 0$. 
Furthermore, the Gordan-Hilbert theorem implies that we can find finitely many
homogeneous $\SL(n)$-invariants $\mathfrak{I}_1,\dots, \mathfrak{I}_N$ of equal degrees 
such that for polystable $f$ and $g$, we have  
\begin{equation}\label{E:distinguish}
\GL(n)\cdot f=\GL(n)\cdot g  
\,\Longleftrightarrow\,
[\mathfrak{I}_1(f): \cdots:\mathfrak{I}_N(f)]=[\mathfrak{I}_1(g): \cdots:\mathfrak{I}_N(g)].
\end{equation}

One can consider a generalization of the above question and ask when two $m$-dimensional linear systems 
$\langle f_1,\dots, f_m\rangle$ and $\langle g_1,\dots, g_m\rangle$ of degree $d$ forms 
are related by a linear change of variables. This again can be phrased 
in terms of a GIT problem, this time given by the action of $\SL(n)$ on $\Grass(m, \Sym^{d} V)$
(or the $\SL(n)$-action on the affine cone over the Grassmannian in its Pl\"ucker embedding).
A priori, to distinguish orbits of this action, one needs to understand polynomial 
$\SL(n)$-invariants on $\wedge^{m} \Sym^{d} V$.
One of the main results of this paper is that for $m=n$, i.e., when the number of forms is equal to the number
of variables, the problem of distinguishing $\SL(n)$-orbits in 
$\Grass(n, \Sym^{d} V)$ can often be reduced to that of distinguishing $\SL(n)$-orbits of degree $n(d-1)$
forms in $n$ (dual!) variables.  

The reason behind this simplification 
is that a generic $n$-dimensional subspace $U\subset \Sym^{d} V$ is spanned by a
regular sequence $g_1,\dots, g_n$. The algebra \[\cA:=k[x_1,\dots,x_n]/(g_1,\dots,g_n)\]
is then 
a graded local Gorenstein Artin algebra of socle degree $n(d-1)$. 
The homogeneous Macaulay inverse system of this algebra
is then an element of $\PP \Sym^{n(d-1)} V^{\vee}$, which we call the associated form of $U$.
A classical theorem of Macaulay says that the associated form morphism
$\AA$ sending $U$ to its associated form is injective
(see Theorem \ref{macaulay-theorem}). 
Alper and Isaev, who initiated a systematic
study of this morphism, showed that 
$\AA$ is a locally closed immersion and 
conjectured that $\AA(U)$ is always semistable and that the induced 
morphism on the GIT quotients is also a locally closed immersion; we refer the reader to 
\cite{alper-isaev-assoc-binary} for details, for the motivation behind these conjectures,
and for a proof in the case of binary forms. In \cite{fedorchuk-ss}, the first author
proved that $\AA(U)$ is indeed semistable for any $U$ spanned by a regular sequence. Here we show:

\begin{theorem}[{Theorem \ref{T:polystability}}]
\label{T:polystabilityintro}
Assume $\chark(k)=0$. Suppose that an element $U \in \Grass(n, \Sym^{d} V)$ is spanned by a regular sequence and is polystable. Then $\AA(U)$ is polystable. 
\end{theorem}

\noindent Consequently, injectivity is preserved on the level of GIT quotients, just as Alper and Isaev conjectured.

As an application of Theorem \ref{T:polystabilityintro}, 
we obtain an invariant-theoretic variant of the Mather-Yau theorem 
for isolated homogeneous hypersurface singularities. The original version of this theorem, 
proved in \cite{mather-yau}, states that an isolated hypersurface singularity in $\CC^n$ is 
determined, up to biholomorphism, by $n$ and the isomorphism class of its moduli (Tjurina) 
algebra. The theorem was extended to the case of non-isolated hypersurface singularities in 
\cite[Theorem 2.26]{greuel-lossen-shustin}. Further, in \cite[Proposition 2.1]{greuel-pham} it 
was established for arbitrary algebraically closed fields of characteristic $0$. Finally, for 
singularities over  algebraically closed fields of arbitrary characteristic, an analogue 
of the theorem was proved in \cite[Theorem 2.2]{greuel-pham}. 
For more details on the history of the Mather-Yau theorem 
we refer the reader to \cite{greuel-pham}.

 The Mather-Yau theorem is non-trivial even for homogeneous singularities and raises a natural question of how exactly a singularity is encoded by the corresponding algebra (see \cite{benson}). 
If $f(x_1,\dots,x_n)=0$ is such a
singularity, defined by a form of degree $d+1$, then its moduli algebra coincides with its Milnor algebra $M_f$, 
which has an associated form $A(f) \in \Sym^{n(d-1)} V^{\vee}$, first studied in \cite{alper-isaev-assoc}.
Our polystability result implies that two forms $f, g\in \Sym^{d+1} V$ define isomorphic isolated
hypersurface singularities (i.e., the completions of the local rings of the hypersurfaces $\{f=0\}$
and $\{g=0\}$ at the origin are isomorphic over the algebraic closure of the field) if and only if
their associated forms $A(f)$ and $A(g)$ map to the same point 
in the GIT quotient $\PP \Sym^{n(d-1)} V^{\vee}\gitq \SL(n)$,
something that can be detected by finitely many homogeneous $\SL(n)$-invariants just as in \eqref{E:distinguish}.
Since the associated form $A(f)$ is computable from the Milnor algebra alone, we
obtain a purely algebraic, and in principle algorithmic, way of deciding when two isolated 
homogeneous hypersurface singularities are isomorphic based solely on their Milnor algebras:
\begin{theorem}[{Theorem \ref{T:invariant-MY}}]
\label{T:invariant-MY-intro}
There exists a finite collection
of homogeneous $\SL(n)$-invariants $\mathfrak{I}_1,\dots, \mathfrak{I}_N$ 
on $\Sym^{n(d-1)} V^{\vee}$ of equal degrees, defined over $k$,  
such that for any two forms $f, g\in \Sym^{d+1} V$ defining isolated
singularities, the two singularities are isomorphic
if and only if
\[
[\mathfrak{I}_1(A(f)): \cdots:\mathfrak{I}_N(A(f))]
=[\mathfrak{I}_1(A(g)): \cdots:\mathfrak{I}_N(A(g))].
\]
\end{theorem}

\subsection*{Notation and conventions}
We work over a field $k$ of characteristic $0$ (not necessarily algebraically closed). 
The dual of a $k$-vector space will be denoted by ${}^\vee$.
Fix $n\geq 1$ and let $V$ be an $n$-dimensional vector space over $k$. 
Let $S:=\Sym V$ be the symmetric algebra on $V$ with the standard grading.

We briefly recall some basic notions of GIT utilized in this paper. Our main reference for GIT is \cite{GIT}, but the reader is also referred, e.g., to \cite[Chapter 9]{lakshmibai-raghavan} for a more elementary exposition that uses modern terminology.
Suppose $W$ is an algebraic representation of a reductive group $G$.
Then $x\in W$ is \emph{semistable} if $0 \notin \overline{G\cdot x}$
and \emph{polystable} if $G\cdot x$ is closed. Similarly,
for $\bar x\in \PP W$, we say that $\bar x$ is semistable (resp., polystable) if some (equivalently, any) lift $x$ of $\bar x$ to $W$
is semistable (resp., polystable). 
The locus of semistable points in $\PP W$ is open and is denoted by $\PP W^{ss}$. 
More generally, if $X\subset\PP W$ is a $G$-invariant projective closed subscheme, then one defines 
the locus of semistable points in $X$ as $X^{ss}:= X\cap\PP W^{ss}$.
The orbits of polystable $k$-points in $X^{ss}$ are in bijection with the $k$-points of the projective
GIT quotient 
\[
X^{ss} \gitq G:= \proj \bigoplus_{m\geq 0} \HH^0\bigl(X, \cO_{X}(m)\bigr)^{G}.
\]
In particular, polystable orbits in $\PP W^{ss}$ (and more generally in $X^{ss}$) 
are distinguished by $G$-invariant 
forms on $\PP W$.  It will be crucial for us that 
in the case of a perfect field, semistability and polystability is determined by the 
standard Hilbert-Mumford numerical criterion; see \cite{kempf} for more details.

Since the definition of the associated form $A(f)$ requires a large enough characteristic
(at the very least, we need $\chark(k)\nmid \deg(f)$ in order for the partial derivatives
of $f$ to form a regular sequence), and our proof of polystability relies on characteristic $0$
results, our Theorems \ref{T:polystabilityintro} and \ref{T:invariant-MY-intro}
require $\chark(k)=0$. The reader can verify that as long as $A(f)$ is defined 
and the field is perfect, our proof of the semistability of $A(f)$ goes through. 
At the moment, we are not aware of any counterexamples to the polystability statement of Theorem \ref{T:polystabilityintro} for fields of (sufficiently large) positive characteristic.

\subsection*{Roadmap of the paper} In Section \ref{S:prelim}, we introduce the main 
actors of this work, the balanced complete intersection algebras and their associated forms, 
and state our principal result (Theorem \ref{T:polystability}).
In Section \ref{S:algebra-prop}, we prove a key technical
commutative algebra proposition.
In Section \ref{S:polystability}, 
we prove Theorem \ref{T:polystability}.
Finally, in Section \ref{S:mather-yau}, we give applications of our preservation of polystability result,
the main of which is an invariant-theoretic variant of the Mather-Yau theorem.

\subsection*{Acknowledgements} During the preparation of this work,
the first author was supported by 
the NSA Young Investigator grant H98230-16-1-0061 and Alfred P. Sloan Research Fellowship,
and the second author was supported by the Australian Research Council. We would like to thank the referees for their very thorough reading of the manuscript and for making many useful suggestions that helped improve the paper.

\section{Associated forms of complete intersections}\label{S:prelim}

\subsection{Gorenstein Artin algebras and Macaulay inverse systems}\label{artinian-gorenstein}

We briefly recall basics of the theory of Macaulay inverse systems of graded Gorenstein 
Artin algebras necessary to state our main result, 
but the reader is encouraged to consult \cite{iarrobino-kanev} for 
a more comprehensive discussion. 

Recall that a homogeneous ideal $I\subset S$ is \emph{Gorenstein} 
if $\cA:=S/I$ is a \emph{Gorenstein Artin $k$-algebra}, 
meaning that $\dim_{k} \cA <\infty$ and $\dim_k \Soc(\cA)=1$.
Here, $\Soc(\cA)$ is the annihilator of the unique maximal ideal $\mi_{\A}$ of $\cA$. 
We endow $\cA$ with the standard grading coming from $S$.
Then 
\[
\cA=\bigoplus_{d=0}^{\nu} \cA_d,
\]
where $\nu$ is the \emph{socle degree} of $\cA$, and $\Soc(\cA)=\cA_{\nu}$. 
The surjection $H_\nu \co S_\nu \twoheadrightarrow \cA_\nu$
is called the $\nu^{th}$ \emph{Hilbert point} of $\cA$, which we regard 
as a point in $\PP S_\nu^{\vee}$.
As we will see shortly, it is dual to the homogeneous Macaulay inverse system of $\cA$. 

We can regard $S=\Sym V$ as a ring of polynomial differential operators
on a `dual ring' $D:=\Sym V^{\vee}$ as follows. Let $x_1,\dots, x_n$ be a basis of $V$ 
and $z_1,\dots, z_n$ be the dual basis of $V^{\vee}$.
Then we have an \emph{apolarity action} of $S$ on $D$ 
\[
\circ\colon S  \times D \to D \label{E:contraction-action}
\]
given by differentiation
\[
g(x_1,\dots,x_n)\circ f(z_1,\dots, z_n):=g(\partial /\partial z_1, \dots, \partial /\partial z_n) f(z_1,\dots,z_n).\label{E:contraction}
\]
Since $\chark(k)=0$, the restricted pairing $S_{d}  \times D_{d} \to k$
is perfect and so defines an isomorphism
\begin{equation}\label{E:duality}
D_{d} \simeq S_{d}^{\vee}
\end{equation}
(see \cite[Appendix A, Example A.5]{iarrobino-kanev} and \cite[Proposition 2.8]{jelisiejew} for more details).

Recall now the following classical result, whose modern 
exposition can be found in 
\cite[Lemmas 2.12 and 2.14]{iarrobino-kanev} or \cite[Exercise 21.7]{eisenbud}.
\begin{theorem}[{Macaulay's theorem \cite[Chapter IV]{macaulay}}]
\label{macaulay-theorem}
For every non-zero $f \in D_{\nu}$, the homogeneous ideal
\[
f^{\perp}:=\{g\in S \mid g\circ f=0\}
\] is such that $S/f^{\perp}$ is a Gorenstein Artin $k$-algebra of socle degree $\nu$. Conversely, for every homogeneous
Gorenstein ideal $I\subset S$ such that $S/I$ has socle
degree $\nu$, there exists $f\in D_{\nu}$ such that $I=f^{\perp}$.
Moreover, $f_1^{\perp}=f_2^{\perp}$ if and only if $f_1$ and $f_2$
are scalar multiples of each other.
\end{theorem} 
\begin{definition}
If $I\subset S$ is a Gorenstein ideal and $\nu$ is the socle degree of the algebra $\cA=S/I$,
then a \emph{(homogeneous) Macaulay inverse system} of $\cA$ is an element 
$f \in D_{\nu}$, given by the above theorem, such that $I=f^{\perp}$. 
\end{definition}
Clearly, all Macaulay inverse systems are mutually proportional and 
the line $\langle f\rangle\in \PP D_{\nu}$
maps to the $\nu^{th}$ Hilbert point $H_{\nu}\in \PP S_{\nu}^{\vee}$ of $\cA$ under isomorphism \eqref{E:duality} for $d=\nu$. This leads to the following useful consequence
of Theorem \ref{macaulay-theorem}:
\begin{corollary}[{cf. \cite[Lemmas 2.15 and 2.17]{iarrobino-kanev}}]\label{L:gorenstein} Let $I$ and $J$ be homogeneous ideals
in $S$ such that $S/I$ and $S/J$ are Gorenstein Artin $k$-algebras of socle
degree $\nu$. Then:
\begin{enumerate}
\item $I_d=\{g \in S_d \mid hg \in I_{\nu} \text{ for all $h\in S_{\nu-d}$}\}$, for $1\leq d\leq \nu$.
\item $I=J$ if and only if $I_{\nu}=J_{\nu}$.
\end{enumerate}
\end{corollary}
\begin{proof} 
(1) implies (2) and follows immediately from Macaulay's theorem by noting that 
for $\cA=S/I$ the pairing $\cA_d \times \cA_{\nu-d} \to \cA_\nu$
is perfect for every $d\leq \nu$. 
\end{proof}

For any $\omega\in S_{\nu}^{\vee}$ with $\ker\omega=I_{\nu}$, papers \cite{eastwood-isaev1,eastwood-isaev2} introduced an \emph{associated form} of $\cA$ as the element of $D_{\nu}$ given by
the formula 
\begin{equation}\label{E:formula-associated}
f_{\cA,\omega}:=\omega\bigl((x_1z_1+\cdots+x_nz_n)^{\nu}\bigr) \in k[z_1,\dots,z_n]_{\nu}.
\end{equation}
Since $f_{\cA,\omega}^{\perp}=I$,
these associated forms give explicit formulae for the Macaulay inverse systems of $\cA$
(see \cite{isaev-criterion} for more details).

\subsection{Koszul complex}\label{S:koszul} Suppose $m$ is a positive integer.
Recall that for $g_1,\dots, g_{m} \in S_d$, the Koszul complex $K_{\bullet}(g_1,\dots,g_m)$ 
is defined as follows.  Let $e_{1}, \dots, e_m$ be the standard degree $d$ generators of the graded free 
$S$-module $S(-d)^{m}$. Then $K_{\bullet}(g_1,\dots, g_m)$ 
is an $(m+1)$-term 
complex of graded free $S$-modules with  
\[
K_{j}(g_1,\dots,g_m):=\wedge^{j} S(-d)^{m}\,\,\hbox{for $j=1,\dots,m$},  \quad K_{0}(g_1,\dots,g_m):=S,
\]
and the differential $d_j \colon K_{j}(g_1,\dots,g_m) \to K_{j-1}(g_1,\dots,g_m)$ given by
\[
d_j( e_{i_1}\wedge \cdots \wedge e_{i_j}):= \sum_{r=1}^{j} (-1)^{r-1} g_{i_r}  e_{i_1}\wedge \cdots\wedge \widehat{e_{i_r}}\wedge\cdots \wedge e_{i_j}.
\]
Note that $\HH_0(K_{\bullet}(g_1,\dots,g_m))=S/(g_1,\dots,g_m)$. We will use without further comment
basic results about Koszul complexes as developed in \cite[Chapter 17]{eisenbud}.

\subsection{Balanced complete intersections and their associated forms} 
\label{S:ci}
Suppose $d\geq 2$ and $m\leq n$. 
Recall that elements $g_1,\dots, g_m\in S_d$ form a regular sequence in $S$
if and only if any of the following equivalent conditions hold: 
\begin{enumerate}
\item $\codim (g_1,\dots,g_m) = m$.
\item the Koszul complex $K_{\bullet}(g_1,\dots,g_m)$  is a minimal free resolution of $S/(g_1,\dots,g_m)$.
\end{enumerate}
Moreover, if $n=m$, then the above conditions are also equivalent to each of
\begin{enumerate}
\item[(3)] the forms $g_1,\dots, g_n$ have no non-trivial common zero in $\bar k^n$.
\item[(4)] the resultant $\Res(g_1,\dots,g_n)$ is non-zero. 
\end{enumerate}

We now recall the definition of the associated form of a complete intersection as first given 
in \cite{alper-isaev-assoc-binary}.
To begin, 
if $g_1,\dots, g_n \in S_d$ form a regular sequence in $S$, then we call $I:=(g_1,\dots, g_n)$ 
a \emph{complete intersection ideal of type $(d)^n$}, or simply 
a \emph{balanced complete intersection} if the degree $d$ and the number of variables
$n$ are understood; here, \lq\lq balanced\rq\rq\ refers to the fact that $g_1,\dots, g_n$ have the same degree. In this case, we also call the algebra $\cA:=S/I$ a \emph{complete intersection algebra of type $(d)^n$}, or a \emph{balanced complete intersection}.
A complete intersection algebra of type $(d)^n$ is a graded 
Gorenstein Artin $k$-algebra with Hilbert
function \[
\sum_{j\geq 0} \dim_k (\cA_j) t^j=\left(\dfrac{1-t^{d}}{1-t}\right)^{n}
\] 
and so has socle degree $n(d-1)$. 

Let $H_{n(d-1)} \co S_{n(d-1)} \to \cA_{n(d-1)}$ be the $n(d-1)^{th}$ Hilbert point of $\cA$.
Denote by $\Jac(g_1,\dots,g_n)$ the Jacobian $n\times n$ matrix of $g_1,\dots, g_n$,
whose $(ij)^{th}$ entry is $\partial g_i/\partial x_j$.
Then $\cA_{n(d-1)}$
is spanned by $H_{n(d-1)}(\det \Jac(g_1,\dots,g_n))$ (see \cite[p.~187]{scheja-storch}),
and so we can choose an isomorphism $\cA_{n(d-1)} \simeq k$ that sends 
$H_{n(d-1)}(\det \Jac(g_1,\dots,g_n))$ to $1$.
Denote the resulting element of $S_{n(d-1)}^{\vee}=\Hom_{k}(S_{n(d-1)},k)$ by $\omega$. 
Then the form $f_{\cA, \omega} \in D_{n(d-1)}$ given by 
Equation \eqref{E:formula-associated} is called \emph{the associated form of $g_1,\dots, g_n$} 
and is denoted by $\AA(g_1,\dots,g_n)$ (cf. \cite{alper-isaev-assoc-binary}). 
The form $\AA(g_1,\dots,g_n)$ is a homogeneous Macaulay inverse system of $\cA$. We remark that the idea of considering this form goes back to \cite{isaev-kruzhilin}.

We let $\Grass(n,S_d)_{\Res}$ be the affine open subset of $\Grass(n,S_d)$ on which the resultant 
(considered as a section of the corresponding line bundle) 
does not vanish. Alper and Isaev defined the 
\emph{associated form morphism}
\[\label{E:assoc-morphism}
\AA\colon \Grass(n,S_d)_{\Res}\rightarrow \PP D_{n(d-1)}, 
\]
that sends a point $U\in \Grass(n,S_d)_{\Res}$ to 
the line spanned by $\AA(g_1,\dots,g_n)$, where $g_1,\dots, g_n$ is any basis of $U$ (see \cite[Section 2]{alper-isaev-assoc-binary}). By \cite[Lemma 2.7]{alper-isaev-assoc-binary}, the morphism $\AA$ is 
$\SL(n)$-equivariant. The preservation of GIT polystability by $\AA$ is the main object of study in this paper. Our main result is Theorem \ref{T:polystabilityintro}, which we restate as follows:

\begin{theorem}\label{T:polystability}
Suppose $U \in \Grass(n, S_d)_{\Res}$ is polystable. Then $\AA(U)$ is polystable. 
\end{theorem}

While proving Theorem \ref{T:polystability}, we also simplify the proof of the semistability of associated
forms, first obtained by the first author in \cite[Theorem 1.2]{fedorchuk-ss}. We refer
the reader to Theorem \ref{A} for a more technical version of Theorem \ref{T:polystability}
that gives a necessary and sufficient condition for $\AA(U)$ to be stable when $k$ is algebraically closed.

\subsection{Balanced complete intersections and decomposability}
Among all codimension $n$ ideals of $S$ generated in degree $d$,
the balanced complete intersections are distinguished using the following 
simple, but important lemma: 
\begin{lemma}\label{L:dimension-0}
Suppose $J \subset S$ is a codimension $n$ homogeneous ideal generated by $J_d$. 
Then either $J$ is a balanced complete intersection
or  
$
(S/J)_{n(d-1)}=0.
$
\end{lemma}
\begin{proof}
Since $J$ has codimension $n$ in $S$, there exist $r_1,\dots, r_n \in J_d$
that form a regular sequence. Set $Y:=(r_1,\dots,r_n) \subset J$. Then $S/Y$ is a balanced 
complete intersection algebra with socle in degree $n(d-1)$. 
Let $H_{n(d-1)}$ be the $n(d-1)^{th}$ Hilbert point of $S/Y$. We have 
two possibilities:
\begin{enumerate}
\item either $H_{n(d-1)}((s)_{n(d-1)}) =0$ for every $s\in J_d$, in which case $J=Y$ by 
Corollary \ref{L:gorenstein},
\item or there exists $s\in J_d$ such that $H_{n(d-1)}((s)_{n(d-1)}) \neq 0$. 
\end{enumerate}
In the latter case, $(s)_{n(d-1)} \not\subset Y_{n(d-1)}$, so 
that $J_{n(d-1)}$ strictly contains $Y_{n(d-1)}$. Since $Y_{n(d-1)}$ is 
already of codimension $1$ in $S_{n(d-1)}$, we conclude that $J_{n(d-1)}=S_{n(d-1)}$ in this case. 
\end{proof}

\begin{definition}\label{D:decomposable}
We say that $U \in \Grass(n, S_d)_{\Res}$ is \emph{decomposable}
if there is a choice of a basis $x_1,\dots, x_n$ of $S_1$, an integer $1\leq a\leq n-1$, and a basis 
$g_1,\dots, g_n$ of $U$ such that 
$g_{a+1}, \dots, g_n \in k[x_{a+1},\dots, x_n]$. An element $U \in \Grass(n, S_d)_{\Res}$ that is not decomposable will be called \emph{indecomposable}. 
For $U \in \Grass(n, S_d)_{\Res}$, we will also speak about the (in)decomposability of the balanced complete intersection ideal $I:=(U)\subset S$ and the balanced complete intersection algebra $S/I$.
\end{definition}

Decomposable complete intersections have a simple structure described
by the following result:
\begin{prop}\label{P:decomposable} Suppose $U=\langle g_1,\dots, g_n\rangle \in \Grass(n, S_d)_{\Res}$ is
such that, for some $1\leq a\leq n-1$, we have $g_{a+1},\dots,g_n\in k[x_{a+1},\dots,x_n]_{d}$. Then $g_{a+1}, \dots, g_n$ is a regular 
sequence in $k[x_{a+1},\dots,x_n]_{d}$ 
and there exists a regular sequence $g_1', \dots, g_a' \in k[x_1,\dots,x_a]_d$ such that 
the closure of the $\SL(n)$-orbit of $U \in \Grass(n, S_d)_{\Res}$ contains 
$\langle g_1',\dots, g_a', g_{a+1},\dots,g_n\rangle$.
\end{prop}
\begin{proof} The first statement is clear. For the second statement, consider the elements $g_i':=g_i(x_1, \dots, x_a, 0,\dots, 0)$. 
Then $g_{1}',\dots,g_a'$
form a regular sequence in $k[x_{1},\dots, x_a]$.
Let $\mu$ be the 1-PS of $\SL(n)$ acting with weight $-(n-a)$ on $x_1,\dots, x_a$ and with weight $a$ on $x_{a+1},\dots, x_n$. Then 
\[
\lim_{t\to 0} \mu(t)\cdot U =\langle g'_1,\dots, g'_a, g_{a+1}, \dots, g_n\rangle \in \Grass(n, S_d)_{\Res}.
\]
This finishes the proof.
\end{proof}
The following is immediate:
\begin{corollary}\label{C:decomposable}  Suppose a decomposable $U \in \Grass(n, S_d)_{\Res}$ is polystable. Then 
there exists $1\leq a\leq n-1$ and a basis $x_1,\dots, x_n$ of $S_1$ such that 
\[
U=\langle g_1, \dots, g_a, g_{a+1},\dots, g_n\rangle,
\]
where $g_1,\dots, g_a$ is a regular sequence in $k[x_1,\dots,x_a]_d$ and $g_{a+1},\dots, g_n$
is a regular sequence in $k[x_{a+1},\dots,x_n]_d$. 
\end{corollary}
The balanced complete intersections described by the previous corollary are called \emph{direct sums}; we will see in Remark \ref{R:indecomposable-not-direct}
that non-polystable balanced complete intersections are always decomposable,
but are not direct sums of indecomposables.
The associated forms of direct sums are computed as follows. 
\begin{lemma}\label{L:decomposable} Suppose for $1\leq a\leq n-1$, we have that $g_1,\dots, g_a \in k[x_1,\dots,x_a]_{d}$
is a regular sequence and $g_{a+1},\dots,g_n\in k[x_{a+1},\dots,x_n]_{d}$ is a regular sequence.
Then 
\[
\AA(g_1,\dots,g_n)=\binom{n(d-1)}{a(d-1)}
\AA(g_1,\dots,g_a)\AA(g_{a+1},\dots,g_n),
\]
where $\AA(g_1,\dots,g_a)$ is the associated form of $g_1,\dots, g_a$ in 
$k[z_1,\dots,z_a]_{a(d-1)}$, and $\AA(g_{a+1},\dots,g_n)$ is that of 
$g_{a+1},\dots, g_a$ in $k[z_{a+1},\dots,z_n]_{(n-a)(d-1)}$.
\end{lemma}
\begin{proof}
As an element of $S_{n(d-1)}^{\vee}$ under the isomorphism \eqref{E:duality}, 
$\AA(g_1,\dots,g_n)$
is uniquely determined by the property that it vanishes on $(g_1,\dots,g_n)_{n(d-1)}$
and satisfies $$\AA(g_1,\dots,g_n) \bigl(\det \Jac(g_1,\dots,g_n)\bigr) = (n(d-1))!.$$
The claim now follows from
$$
\det \Jac(g_1,\dots,g_n) = \det \Jac(g_1,\dots,g_a) \det \Jac(g_{a+1},\dots,g_n).
$$
\end{proof}

\section{Recognition criterion for decomposable balanced complete intersections}
\label{S:algebra-prop}

Decomposable balanced complete intersections play a crucial role in our inductive 
proof of polystability of associated forms. In this section, we obtain a criterion
for a balanced complete intersection ideal to be decomposable based only on partial 
information about the ideal. Although technical, this result may be of independent interest;
in fact, it has already been used by the first author to give a new 
criterion for forms defining smooth hypersurfaces 
to be of Sebastiani-Thom type \cite{fedorchuk-direct}. 
We note that the results of this section are
valid over an arbitrary field $k$, with no restriction on its characteristic. 
\begin{prop} \label{P:recognition}
 Let $1\le b\le n-1$ and suppose that $I \subset k[x_1,\dots,x_n]$ is a complete intersection
ideal of type $(d)^n$ such that 
\begin{enumerate}
\item[(A)]\label{C1} the homomorphic image of $I$ in $k[x_1,\dots,x_n]/(x_{b+1}, \dots, x_n) \simeq\linebreak k[x_1,\dots,x_b]$
 is a balanced complete intersection ideal, equivalently, \[\dim_k \bigl( I_d \cap (x_{b+1}, \dots, x_n)\bigr)=n-b,\]
\item[(B)]\label{C2} $(x_{b+1}, \dots,x_n)^{(n-b)(d-1)+1} \subset I$. 
\end{enumerate}
Then there are $n-b$ linearly independent elements
$g_{b+1}, \dots, g_n$ in the intersection $I_d \cap k[x_{b+1},\dots, x_n]$; in particular, $I$ is decomposable. 
 \end{prop}
 
 \begin{remark} Note that conditions (A) and (B) are necessary for the conclusion to hold. 
 Indeed, if $g_{b+1}, \dots, g_n \in I_d \cap k[x_{b+1},\dots, x_n]$ is a regular sequence, then 
 $I_d \cap (x_{b+1}, \dots, x_n)=(g_{b+1}, \dots, g_n)$, and $(g_{b+1}, \dots, g_n)$ is a balanced
 complete intersection ideal in $k[x_{b+1},\dots, x_n]$ and so contains $(x_{b+1}, \dots,x_n)^{(n-b)(d-1)+1}$.
 The difficulty lies in verifying the sufficiency of these conditions.
 \end{remark}
 
 \begin{proof}[Proof of  Proposition \ref{P:recognition}]
  Note that for $b=n-1$, Proposition \ref{P:recognition} is obvious, because in this case $x_n^{d} \in I_d$ by (B). 
  So we assume that $b\leq n-2$ in what follows. 
  
  Using condition (A), we can find a basis $g_1,\dots, g_n$ of $I_d$ satisfying $g_{b+1},\dots, g_n \in (x_{b+1},\dots, x_n)$. 
 We will prove that in fact $g_{b+1},\dots, g_n \in k[x_{b+1},\dots, x_n]$.
 We separate our argument into 
two parts, given by two key Lemmas \ref{alg-claim-1} and \ref{alg-claim-2}.

\begin{lemma}\label{alg-claim-1} 
Suppose $I=(g_1,\dots, g_n)$ is a complete intersection ideal of type $(d)^n$ such that 
$g_{b+1},\dots, g_n \in (x_{b+1},\dots, x_n)$ and condition (B) of Proposition \ref{P:recognition} holds. Then 
\[
(x_{b+1}, \dots,x_n)^{(n-b)(d-1)+1}
 \subset (g_{b+1}, \dots, g_n).
\]
\end{lemma}

\begin{remark} The idea behind our proof of this lemma is to understand 
all the syzygy modules of the ideals $(x_{b+1}, \dots,x_n)^{(n-b)(d-1)+1}$ and
$(g_1,\dots, g_n)$. Comparing syzygies of a certain order then gives the requisite statement.
We encourage the reader to keep in mind 
the first non-trivial case given by a regular sequence 
$g_1, g_2, g_3\in k[x_1,x_2,x_3]_d$ such that $g_2, g_3\in (x_2,x_3)$ and 
$(x_{2}, x_3)^{2d-1} \subset (g_1,g_2,g_3)$. The lemma asserts in this case that 
in fact $(x_{2}, x_3)^{2d-1}\subset (g_2,g_3)$.  \end{remark}

\begin{proof}[Proof of Lemma \ref{alg-claim-1}]
Set $N:=(n-b)(d-1)+1$,
and $J:=(x_{b+1},\dots, x_n)^{N}$. 
Let $\tilde{J}:=J\cap k[x_{b+1},\dots, x_n]$ and $R:=k[x_{b+1},\dots, x_n]$.
Then $\tilde{J}$
is the $N^{th}$ power of the irrelevant ideal in $R$ and so has regularity
$N$, for example, by the computation of its local cohomology
(cf.~\cite[Lemma 1.7]{bayer-stillman}).
Hence $\tilde J$ has a \emph{linear} minimal free
resolution as an $R$-module. In fact, as explained in
\cite[pp.~269-270]{buchsbaum-eisenbud}, 
an explicit minimal free resolution of $\tilde J$ was constructed by Buchsbaum and Rim
using the Eagon-Northcott complex \cite{buchsbaum-rim}. It follows that the
minimal free resolution of $\tilde J$ has the following form:
\begin{equation}\label{res-1}
\begin{array}{l}
0 \to R^{\ell_{n-b}}(-N-n+b+1) \to \cdots\\
\hspace{3cm} \cdots \to R^{\ell_2}(-N-1)\to R^{\ell_1}(-N) \to \tilde{J} \to 0.
\end{array}
\end{equation}
Since $S$ is a flat $R$-algebra, tensoring by $S$ we obtain a minimal
free resolution of $J$ as an $S$-module:
\begin{equation*}
0 \to S^{\ell_{n-b}}(-N-n+b+1) \to\cdots 
\to S^{\ell_2}(-N-1)\to S^{\ell_1}(-N) \to J \to 0.
\end{equation*}

Consider now the Koszul complex $K_{\bullet}(g_1,\dots,g_n)$,
which gives a minimal free resolution of $S/I$; we keep the notation of \S\ref{S:koszul}.
By our assumption, we have an inclusion $J \subset I$. It gives rise to a map of complexes
\begin{equation}\label{complex-map1}
\makebox[250pt]{$\begin{gathered}
\xymatrix{
S^{\ell_{n-b}}(-N-n+b+1)  \ar[r]\ar[d]^{m_{n-b}} & \cdots \ar[r]\ar[d]
&
S^{\ell_1}(-N)
\ar[r]\ar[d]^{m_{1}}  & S \ar[d]^{m_0=\operatorname{Id}_S} \\
K_{n-b}(g_1,\dots,g_n)  \ar[r] & \cdots \ar[r] &
K_{1}(g_1,\dots,g_n)  \ar[r]  & S.
}
\end{gathered}$}
\end{equation}

Next, note that the Koszul complex $K_{\bullet} (g_1,\dots,g_n)$ contains
$K_{\bullet}(g_{b+1},\dots,g_n)$ as a subcomplex. Let $Q_{\bullet}$ be the quotient
complex. Then
\[
Q_{i}:=K_{i}(g_1,\dots,g_n)/K_{i}(g_{b+1},\dots,g_n)\]
is a free $S$-module for every $i\geq 1$, and from the long exact sequence
in homology associated to the short exact sequence of complexes
\[
0 \to K_{\bullet}(g_{b+1},\dots,g_n) \to K_{\bullet} (g_1,\dots,g_n) \to Q_{\bullet} \to 0,
\]
we obtain that
$Q_{0}=0$, $\HH_1(Q_{\bullet})=I/(g_{b+1},\dots,g_n) \subset S/(g_{b+1},\dots,g_n)$, and
$\HH_{i}(Q_{\bullet})=0$ for $i>1$.

Composing \eqref{complex-map1} with the quotient morphism, 
and replacing $Q_0$ by the quotient $S/(g_{b+1},\dots,g_n)$, we obtain
a map of exact complexes 
\begin{equation}\label{complex-map-2}
\makebox[250pt]{$\begin{gathered}
\xymatrix{
S^{\ell_{n-b}}(-N-n+b+1)  \ar[r]\ar[d]^{\tilde m_{n-b}} & \cdots
\ar[r]\ar[d] &  
S^{\ell_1}(-N)
\ar[r]\ar[d]^{\tilde m_{1}}  & S \ar[d]^{\tilde m_0} \\
Q_{n-b} \ar[r] & \cdots \ar[r] & 
Q_{1} \ar[r]^{d_1 \qquad } & S/(g_{b+1},\dots,g_n).
}
\end{gathered}$}
\end{equation}
Note that $J\subset (g_{b+1},\dots,g_n)$ if and only if 
\[
\tilde m_0(J)
=\operatorname{Im}(d_1 \circ \tilde m_1)=0,
\] which is
what we are going to prove. We begin with the following: 
\begin{claim}\label{zero-map}  $\tilde m_{n-b}=0$,
or, equivalently, $\operatorname{Im}(m_{n-b})\subset K_{n-b}(g_{b+1},\dots,g_n)$.
\end{claim}
\begin{proof}

The key observation is that \[-N-n+b+1=-(n-b)d.\] Since
$
K_{n-b}(g_1,\dots,g_n) \simeq S^{\binom{n}{n-b}}\bigl(-(n-b)d\bigr),
$
it follows that $m_{n-b}$ in \eqref{complex-map1} 
is given by a matrix of scalars. 
Hence it suffices to prove that 
\[
\operatorname{Im}(m_{n-b}\otimes_{S} \bar S)\subset K_{n-b}(g_{b+1},\dots,g_n)\otimes_{S} \bar S,\] where
$\bar S:=S/(x_{b+1}, \dots, x_n)$. 
Upon tensoring \eqref{complex-map1} with $\bar S$, all differentials in the 
top row become zero because \eqref{res-1} was a minimal resolution of an ideal in $k[x_{b+1},\dots, x_n]$. It follows that 
$\operatorname{Im}(m_{n-b}\otimes_{S} \bar S) \subset \ker (d_{n-b}\otimes_{S} \bar S)$. Since $\operatorname{Im}(m_{n-b})$ is generated in graded degree $(n-b)d$, 
we at last reduce to showing that 
$\ker (d_{n-b}\otimes_{S} \bar S)_{(n-b)d} \subset (K_{n-b}(g_{b+1},\dots,g_n)\otimes_{S} \bar S)_{(n-b)d}$.
Let 
$\bar g_i$ be the image of $g_i$ in $\bar S$ for $i=1,\dots, n$; we have
$\bar g_{b+1}=\cdots=\bar g_n=0$ by our assumption. 
Then
\begin{equation}
\begin{array}{l}
K_{n-b}(g_1,\dots, g_n)\otimes_{S} \bar S \simeq K_{n-b}(\bar g_1,\dots \bar g_b, 0, \dots, 0) \\
\vspace{-0.3cm}\\
\hspace{3cm}\simeq \bigoplus_{j=0}^{n-b} K_{j}(\bar g_1,\dots,\bar  g_b)\otimes_{k} K_{n-b-j}(0,\dots, 0).
\end{array}\label{tensor-product}
\end{equation}
Also, the inclusion $K_{n-b}(g_{b+1},\dots,g_n)\subset K_{n-b}(g_1,\dots,g_n)$
induces the isomorphism $K_{n-b}(g_{b+1},\dots,g_n)\otimes_{S} \bar S 
\simeq K_{0}(\bar g_1,\dots, \bar g_b)_{0}\otimes_{k} K_{n-b}(0,\dots, 0)_{(n-b)d}$. 
In \eqref{tensor-product}, $d_{n-b}\otimes_{S}\bar S$ restricts to $d_j\otimes 1$ on the summand
$K_{j}(\bar g_1,\dots, \bar g_b)\otimes_{k} K_{n-b-j}(0,\dots, 0)$. In graded degree $(n-b)d$,
this restriction is simply 
\[
\begin{array}{l}
d_j\otimes 1\colon K_{j}(\bar g_1,\dots, \bar g_b)_{dj} \otimes_k k^{\binom{n-b}{n-b-j}} 
\to\\
\hspace{5cm} K_{j-1}(\bar g_1,\dots, \bar g_b)_{dj} \otimes_k  k^{\binom{n-b}{n-b-j}} 
\end{array}
\] 
This map is injective for $n-b\geq j\geq 1$ since the
matrices defining the Koszul differentials in $K_{\bullet}(\bar g_1,\dots, \bar g_b)$ are of full rank, as $\bar g_1,\dots, \bar g_b$ are linearly independent over $k$. The claim follows.
\end{proof}

We now proceed to prove that $\tilde m_{n-b}=0$ implies 
$d_1 \circ \tilde m_1=0$. 
To lighten notation, we let $T_{\bullet}$ be the exact complex
given by the top row in diagram \eqref{complex-map-2}, so that $T_i:=S^{\ell_i}(-N-i+1)$ for 
$i=1,\dots, n-b$ and $T_0:=S$. 
\begin{claim}\label{null-homotopy} 
 The map of complexes $\tilde{m}\colon T_{\bullet} \to Q_{\bullet}$ is null-homotopic in homological degree
$i\leq n-b$.  Namely, for $i=0,\dots, n-b-1$, there exist maps 
$h_i\colon T_i \to Q_{i+1}$ such that 
\[
\tilde{m}_{i}=d_{i+1}\circ h_i + h_{i-1} \circ d_i \,\, \text{for every $i=1,\dots, n-b-1$}.
\] 
\end{claim}
\begin{proof}
It suffices to see that the dual map $\tilde{m}^{\vee} \colon \Hom_S(Q_{\bullet}, S) \to \Hom_S(T_{\bullet}, S)$
is null-homotopic. Note that the top row of \eqref{complex-map-2} is the resolution of the $S$-module $S/J$.
Since $\operatorname{Ext}^{j}_S(S/J,S)$ vanishes for $j<\codim J=n-b$, the complex $\Hom_S(T_{\bullet}, S)$ gives a resolution of $\Hom_S(T_{n-b},S)$.
Namely,
\[
0 \to \Hom_S(T_0, S) \to \Hom_S(T_1,S) \to \cdots \to
\Hom_S(T_{n-b},S)
\]
is exact. 
Since $\tilde{m}_{n-b}^{\vee}=0$ by Claim \ref{zero-map}, 
it follows by, e.g., \cite[Porism 2.2.7]{weibel}, that $\tilde{m}^{\vee}$ is null-homotopic.
\end{proof}

To finalize the proof of Lemma \ref{alg-claim-1}, it remains to observe that the element $h_0\in \Hom_S(S, Q_1)$
from Claim \ref{null-homotopy} 
must be zero because $h_0$ is a homomorphism of graded $S$-modules and $Q_1=S(-d)^{b}$.
It follows that $\tilde m_{1}= d_{2}\circ h_{1} + h_{0}\circ d_{1}=d_{2}\circ h_{1}$ and so
\[
d_1\circ \tilde m_{1}=d_1 \circ d_{2}\circ h_{1}=0,
\]
as desired.
\end{proof}

\begin{lemma}\label{alg-claim-2} Suppose $(g_1,\dots, g_n)$ is a complete intersection
ideal of type $(d)^n$ in $k[x_1,\dots,x_n]$ such that one has $(g_{b+1}, \dots, g_n) \subset (x_{b+1}, \dots,x_n)$ and\linebreak
$(x_{b+1}, \dots,x_n)^{(n-b)(d-1)+1}
 \subset (g_{b+1}, \dots, g_n).$ Then
$$
g_{b+1}, \dots, g_n\in k[x_{b+1},\dots, x_n].
$$
\end{lemma}
\begin{proof}[Proof of Lemma \ref{alg-claim-2}]

Since $(x_{b+1}, \dots,x_n)^{(n-b)(d-1)+1}
 \subset (g_{b+1}, \dots, g_n)$, it follows that
\[
g'_j:=g_j(0,\dots, 0, x_{b+1}, \dots, x_{n}), \quad \text{$j=b+1,\dots, n$,}
\]
form a regular sequence in $k[x_{b+1},\dots,x_n]$. 

Consider the balanced complete intersection ideals
\begin{align}
J  &:=(x_1^d,\dots, x_{b}^{d}, g_{b+1},\dots,g_{n}),\nonumber \\
J' &:=(x_1^d,\dots, x_{b}^{d}, g'_{b+1},\dots,g'_{n}),\nonumber
\end{align}
each of which is generated by a regular sequence in $S$. To establish the lemma, it suffices to show that $J=J'$, which by Corollary \ref{L:gorenstein}(2) is equivalent to 
$J_{n(d-1)}=J'_{n(d-1)}$. 

Since $\dim_k J_{n(d-1)}=\dim_k J'_{n(d-1)}=\dim_k S_{n(d-1)}-1$, 
we only need to prove that $J'_{n(d-1)}\subset J_{n(d-1)}$, 
and in fact that \[(g'_{b+1},\dots,g'_{n})_{n(d-1)}\subset J_{n(d-1)}.\]
Fix an element \[\sum_{i=b+1}^n c_i g'_i \in (g'_{b+1},\dots,g'_{n})_{n(d-1)},\] where
$c_i \in S_{n(d-1)-d}$. 
Since $x_1^d,\dots,x_b^d\in J$
and \[k[x_{b+1}, \dots,x_n]_{(n-b)(d-1)+1} \subset J,\] we can assume that 
$c_i \in x_1^{d-1}\cdots x_b^{d-1} k[x_{b+1},\dots,x_n]_{(n-b)(d-1)-d}$.
As we have $g'_i \equiv g_i \pmod{(x_1,\dots,x_{b})}$, it follows
that
\[
\sum_{i=b+1}^n c_i g'_i \equiv \sum_{i=b+1}^n c_i g_i \pmod{(x_1^d,\dots, x_b^d)},\]
which shows that $\sum_{i=b+1}^n c_i g'_i\in J$ as required.
\end{proof}
This concludes the proof of Proposition \ref{P:recognition}.

\end{proof}

\section{Preservation of polystability}
\label{S:polystability}

In this section, we prove Theorem \ref{T:polystability}, which is the main result of this paper.
Take $U\in\Grass(n,S_d)_{\Res}$ and let $I:=(g \mid g\in U)\subset S$. Since the field
$k$ is perfect, we can use the Hilbert-Mumford numerical criterion to analyze the GIT stability 
of $\AA(U)$. So for any non-trivial 1-PS $\rho$ of $\SL(n)$ we choose a
basis $x_1,\dots, x_n$ of $S_1$ on which $\rho$ acts diagonally with weights 
$w_1\leq w_2\leq \cdots \leq w_n$. Note that $\rho$ acts with opposite weights on the dual
basis $z_1,\dots, z_n$ of $D_1$. 
To apply the numerical criterion to the form
$\AA(U)\in \PP \bigl(k[z_1,\dots,z_n]_{n(d-1)}\bigr)$, we observe that by (\ref{E:formula-associated}) a monomial $z_1^{d_1}\cdots z_n^{d_n}$
of degree $n(d-1)$ appears with a non-zero coefficient in $\AA(U)$ 
if and only if $x_1^{d_1}\cdots x_n^{d_n} \notin I_{n(d-1)}$. The following lemma allows us to produce
such monomials. Before stating the lemma, recall
that $x_1^{a_1}\dots x_n^{a_n}<_{grevlex} x_1^{b_1}\dots x_n^{b_n}$ if and only 
if either $\sum_{i=1}^n a_i <  \sum_{i=1}^n b_i$ or $\sum_{i=1}^n a_i =\sum_{i=1}^n b_i $ and the last non-zero entry of the vector $(a_1,\dots,a_n)-(b_1,\dots,b_n)$ is positive. 

\begin{lemma}[Grevlex Lemma]\label{grevlex} Fix $1\leq a\leq n$ and $N\geq 0$. 
Suppose 
\[
\bigl((x_{a+1},\dots, x_{n})^N\bigr)_{n(d-1)} \not\subset I_{n(d-1)}.
\]
Let $M=x_1^{d_1}\cdots x_n^{d_n}$ be the smallest with respect to $<_{grevlex}$ monomial
that belongs to $\bigl((x_{a+1},\dots, x_{n})^N\bigr)_{n(d-1)}$ and that does not lie in $I_{n(d-1)}$.
Then for every $i=1,\dots, a$ we have 
\[
d_1+\cdots+d_i \leq i(d-1).
\]
In particular, taking $a=n$, we conclude that if $M=x_1^{d_1}\cdots x_n^{d_n}$ is the smallest with respect to $<_{grevlex}$ monomial
in $k[x_1,\dots,x_n]_{n(d-1)}\setminus I_{n(d-1)}$, then for every $i=1,\dots, n$ we have 
\[
d_1+\cdots+d_i \leq i(d-1).
\]
\end{lemma}
\begin{proof} We only use the fact that $I$ is generated in degree $d$ and has 
codimension $n$ in $S$. 
By way of contradiction, suppose that $d_1+\cdots+d_i > i(d-1)$ for some $i\leq a$. Choose a basis $g_1,\dots,g_n$ in $U$ and
let $J\subset k[x_1,\dots,x_i]$ be the ideal generated by the forms
$g_j(x_1,\dots,x_i,0,\dots,0)$, for $j=1,\dots, n$. Then 
\[
\dim k[x_1,\dots,x_i]/J=\dim S/(g_1,\dots,g_n,x_{i+1},\dots,x_n)=0.
\]
Hence by Lemma \ref{L:dimension-0}, $(k[x_1,\dots,x_i]/J)_{K}=0$ for all $K>i(d-1)$.
Thus
$x_1^{d_1}\cdots x_i^{d_i} \in J$,
and so
$x_1^{d_1}\cdots x_i^{d_i} \in I + (x_{i+1}, \dots, x_n)$. 
We conclude that
\[
M=x_1^{d_1}\cdots x_i^{d_i} x_{i+1}^{d_{i+1}}\cdots x_n^{d_n} 
\in I + (x_{i+1}, \dots, x_n)x_{i+1}^{d_{i+1}}\cdots x_n^{d_n}.
\]
As $M \notin I_{n(d-1)}$, there is a monomial $M' \in \bigl((x_{i+1}, \dots, x_n)x_{i+1}^{d_{i+1}}\cdots x_n^{d_n}\bigr)_{n(d-1)}$
that does not lie in $I_{n(d-1)}$ either. Since $i\leq a$, we have
$M'\in (x_{a+1},\dots, x_{n})^N$. 
However, $M'<_{grevlex} M$, which contradicts our choice of $M$.\end{proof}

\subsection{Proof of semistability}
\label{SS:semistability}
Let
$M=x_1^{d_1}\cdots x_n^{d_n}$ be the smallest with respect to $<_{grevlex}$ monomial of degree $n(d-1)$ that does not lie in $I_{n(d-1)}$. Then $M^{\vee}:=z_1^{d_1}\cdots z_n^{d_n}$ appears with a non-zero coefficient in $\AA(U)$.
By Lemma \ref{grevlex}, we have 
\[
d_1+\cdots+d_i \leq i(d-1), \quad \text{for all $1\leq i \leq n$}.
\]
Hence the $\rho$-weight of $z_1^{d_1}\cdots z_n^{d_n}$ satisfies
\begin{multline*}
-\sum w_id_i =  \sum_{i=1}^{n} (w_{i+1}-w_{i})(d_1+\cdots+d_i) \quad \text{(here we set $w_{n+1}:=0$})\\
\leq \sum_{i=1}^{n} (d-1) i(w_{i+1}-w_{i})=-(d-1)\sum_{i=1}^n w_i=0. 
\end{multline*}
The Hilbert-Mumford numerical criterion then implies that $\AA(U)$ is semistable.

\subsection{Proof of polystability: decomposable case}

To prove Theorem  \ref{T:polystability},
we proceed by induction on $n$. The base case is $n=1$, where the statement is obvious 
because the only balanced complete intersection ideal for $n=1$ is $(x_1^d) \subset k[x_1]$ 
and the corresponding associated form is $z_1^{d-1}$, up to a non-zero scalar. 

Suppose that the theorem is established for all positive integers less than a given $n\ge 2$
and $U\in \Grass(n, S_d)_{\Res}$ is polystable. If $U$ is decomposable, then by Corollary \ref{C:decomposable},
we can assume that for some $1\leq a\leq n-1$, we have a decomposition 
$U=U_1\oplus U_2$, where $U_1\in \Grass(a, k[x_1,\dots,x_a]_d)_{\Res}$ 
and $U_2\in \Grass(n-a, k[x_{a+1},\dots,x_n]_d)_{\Res}$. By Lemma \ref{L:decomposable}, we have
\[
\AA(U)=\AA(U_1)\AA(U_2).
\]
Since $U_1$ and $U_2$ are polystable with respect to $\SL(a)$ and $\SL(n-a)$ actions, respectively, 
the induction hypothesis and the following 
standard result finalizes the proof in the case of a decomposable $U$:
\begin{lemma}\label{L:luna} Let $V=V_1 \oplus V_2$, with $n_i:=\dim V_i \geq 1$.
Suppose $F_1\in \Sym^{d_1} V_1$ and  $F_2\in \Sym^{d_2}V_2$ are both non-zero, where 
$n_1d_2=n_2d_1$.
Then $F:=F_1F_2$ considered as an element of $\Sym^{d_1+d_2}V$ 
is $\SL(V)$-polystable if $F_i$ is $\SL(V_i)$-polystable for each $i$.
\end{lemma}
\begin{proof}
Let $\lambda$ be the one-parameter subgroup of $\SL(V)$ such that $V_1$ is the
weight space of $\lambda$ with weight $-n_2$ and $V_2$ is the weight space
of $\lambda$ with weight $n_1$. Then $\lambda$ stabilizes $F$ by the
assumption $n_1d_2=n_2d_1$. 
The centralizer of $\lambda$ in $\SL(V)$ is 
\[
C_{\SL(V)}(\lambda)=\bigl(\GL(V_1)\times \GL(V_2)\bigr) \cap \SL(V).
\] Since $\chark(k)=0$, \cite[Corollary 4.5(a)]{kempf} (see also \cite[Corollaire 2 and Remarque 1]{luna-adherences})
applies, and so the $\SL(V)$-orbit of $F$ is closed if the $C_{\SL(V)}(\lambda)$-orbit of $F$ is closed. 
However, by the assumption $n_1d_2=n_2d_1$, every element of the center of $\bigl(\GL(V_1)\times \GL(V_2)\bigr) \cap \SL(V)$ acts on $F$ as multiplication by a root of unity.
It follows that the $C_{\SL(V)}(\lambda)$-orbit of $F$ is closed
if and only if the $\SL(V_1)\times \SL(V_2)$-orbit of $F$ is closed, i.e., if and only if the $\SL(V_i)$-orbit of $F_i$ is closed for $i=1,2$.
\end{proof}

\subsection{Proof of polystability: indecomposable case}

Suppose that an element $U\in \Grass(n,S_d)_{\Res}$ is indecomposable. We will use the notation and keep in mind the conclusion of \S\ref{SS:semistability}. Assume that for some $\rho$ the limit $\lim_{t\to 0} \rho(t)\cdot \overline{\AA(U)}$ exists, where $ \overline{\AA(U)}$ is a lift of $\AA(U)$ to $D_{n(d-1)}$.
This implies that $w_{\rho}(M)=0$, where $w_{\rho}(M)$ is the $\rho$-weight of $M=x_1^{d_1}\cdots x_n^{d_n}$. From this we deduce a number of preliminary results.   
\begin{lemma}\label{L:weights} 
Let $1\le a\leq n-1$ be the index such that $w_{a+1}=\cdots=w_n$ and $w_a<w_{a+1}$. Then $d_1+\cdots+d_a=a(d-1)$. 
\end{lemma}
 \begin{proof}
Assuming that $d_1+\cdots+d_a\leq a(d-1)-1$, we see 
\[
0=w_{\rho}(M)=\sum_{i=1}^n d_i w_i > \left(\sum_{i=1}^{a-1} d_i w_i\right) + (d_a+1)w_a +\left(\sum_{i=a+1}^n d_i-1\right)w_n \geq 0,
\]
which is impossible.
 \end{proof}
 
\begin{lemma}\label{L:complete} 
Let $a$ be the integer introduced in Lemma \ref{L:weights}.
Then the homomorphic image of $I$ in $k[x_1,\dots,x_n]/(x_{a+1},\dots,x_n)\simeq k[x_1,\dots,x_a]$ 
is a balanced complete intersection ideal of type $(d)^a$ in $k[x_1,\dots,x_a]$.
\end{lemma}
\begin{proof}
Denote the image ideal by $J$ and suppose that $J$ is not a complete intersection ideal. 
Then by Lemma \ref{L:dimension-0} we have $J_{a(d-1)}=(x_1,\dots,x_a)_{a(d-1)}$, 
and Lemma \ref{L:weights} implies
$x_1^{d_1}\cdots x_a^{d_a} \in J_{a(d-1)}$. This means
\[
x_1^{d_1}\cdots x_a^{d_a} \in I_{a(d-1)} + (x_{a+1},\dots, x_n).
\]
But then 
\[
M=x_1^{d_1}\cdots x_a^{d_a} x_{a+1}^{d_{a+1}}\cdots x_n^{d_n} 
\in I + (x_{a+1}, \dots, x_n)x_{a+1}^{d_{a+1}}\cdots x_n^{d_n}.
\]
Since $M\not\in I_{n(d-1)}$, there is a monomial 
\[
M' \in ((x_{a+1}, \dots, x_n)x_{a+1}^{d_{a+1}}\cdots x_n^{d_n})_{n(d-1)}
\]
that does not lie in $I_{n(d-1)}$ either. However, $M'<_{grevlex} M$, which is a contradiction.
\end{proof}

\begin{lemma}\label{L:B-cond} We have 
\[\label{e-conjecture}
k[x_{a+1}, \dots,x_n]_{(n-a)(d-1)+1}\subset I_{(n-a)(d-1)+1}.
\]
\end{lemma}
\begin{proof} Since $I$ is a homogeneous ideal
in $S$ such that $S/I$ is a Gorenstein Artin $k$-algebras of socle
degree $n(d-1)$, 
by Corollary \ref{L:gorenstein}(1), it suffices to show that 
$\left((x_{a+1}, \dots,x_n)^{(n-a)(d-1)+1}\right)_{n(d-1)}\subset I_{n(d-1)}$.
Assume the opposite and let $L=x_1^{c_1} x_2^{c_2} \cdots x_a^{c_a} x_{a+1}^{c_{a+1}}\cdots x_n^{c_n}$ be the smallest with respect to $<_{grevlex}$ monomial in $\left((x_{a+1}, \dots,x_n)^{(n-a)(d-1)+1}\right)_{n(d-1)}$ that is not in $I_{n(d-1)}$. 
Then by Lemma \ref{grevlex}, for every $i\leq a$, we have
$c_1+\cdots+c_{i} \leq i(d-1)$. Moreover, $c_1+\cdots+c_a \leq a(d-1)-1$ by assumption.
But then 
\[
w_{\rho} (L)>\sum_{i=1}^{a-1}c_iw_i+(c_a+1)w_a+(n-a)(d-1)w_n\ge 0.
\]
As $L$ has positive $\rho$-weight, it must lie in $I_{n(d-1)}$ by the assumption that the limit $\lim_{t\to 0} \rho(t)\cdot \overline{\AA(U)}$ exists, which contradicts our choice of $L$.
\end{proof}

Lemmas \ref{L:complete} and \ref{L:B-cond} imply that both conditions (A) and (B) 
of Proposition \ref{P:recognition} are satisfied. Hence $U$ is decomposable, contradicting 
our assumption. This proves that for every one-parameter subgroup $\rho$ of $\SL(n)$ the limit $\lim_{t\to 0} \rho(t)\cdot \overline{\AA(U)}$ does not exist. By the Hilbert-Mumford numerical criterion we then see that $\AA(U)$ is polystable. 
\qedsymbol

We note that our proof in fact gives a more technical version of Theorem \ref{T:polystability}.
\begin{theorem}\label{A}
Suppose $U\in \Grass(n,S_d)_{\Res}$. If $U$ is indecomposable, then for every one-parameter
subgroup $\rho$ of $\SL(n)$ defined over $k$ the limit $\lim_{t\to 0} \rho(t)\cdot \overline{\AA(U)}$ does not exist. 
In particular, $\AA(U)$ is polystable. Furthermore, 
if $U$ is indecomposable over $\bar k$, then $\AA(U)$ is stable.
\end{theorem}

Notice that, over a non-closed field, the indecomposability of an element $U\in \Grass(n,S_d)_{\Res}$ does not imply on its own that $\AA(U)$ is stable. Indeed, if $U$ is indecomposable over $k$, it is possible for $U$ to be decomposable over $\bar k$. For example, working over $\CC$,
\[
\AA\bigl((x_1+ix_2)^{d}+(x_1-ix_2)^{d}, i((x_1+ix_2)^{d}-(x_1-ix_2)^{d})\bigr)=\langle (z_1^2+z_2^2)^{d-1}\rangle
\]
is not stable while the balanced complete intersection is defined and indecomposable over $\RR$.

We also note that Theorem \ref{A} has a curious consequence for the classification of 
polystable points in $\Grass(n,S_d)_{\Res}$:
\begin{corollary}\label{C:polystable-bci} 
 Suppose $U\in \Grass(n,\Sym^{d} V)_{\Res}$. Then $U$ is polystable with respect
to $\SL(V)$-action if and only if
there is a decomposition $V=\bigoplus_{i=1}^m V_i$, where $m\geq 1$ and $n_i:=\dim_k V_i \geq 1$, 
and indecomposable $U_i \in \Grass(n_i, \Sym^{d} V_i)_{\Res}$ such that 
\begin{equation}\label{E:direct-sum}
U=\bigoplus_{i=1}^m U_i.
\end{equation}
\end{corollary}
\begin{proof}
Applying Corollary \ref{C:decomposable} repeatedly, we see that a polystable $U$ is  
of the form given by Equation \eqref{E:direct-sum} with each $U_i$ indecomposable.

It remains to show that every $U$ with such decomposition is polystable. By 
Theorem \ref{A} each $\AA(U_i)$ is polystable, and since $\AA(U)=\AA(U_1)\cdots \AA(U_m)$ by Lemma \ref{L:decomposable}, the polystability 
of $\AA(U)$ follows by Lemma \ref{L:luna}. Since $\AA$ is an $\SL(V)$-equivariant locally closed
immersion by \cite[\S2.5]{alper-isaev-assoc-binary}, this
implies the polystability of $U$.
\end{proof}
\begin{remark}\label{R:indecomposable-not-direct} It follows from Corollary \ref{C:polystable-bci} that every non-polystable
element $U \in \Grass(n,\Sym^{d} V)_{\Res}$ is necessarily decomposable, but 
cannot be written as a direct sum $\bigoplus_{i=1}^m U_i$, where
$V=\bigoplus_{i=1}^m V_i$ and the elements $U_i \in \Grass(\dim V_i, \Sym^{d} V_i)_{\Res}$
are indecomposable.
\end{remark}

\section{Invariant-theoretic variant of the Mather-Yau theorem}
\label{S:mather-yau}
As before, we continue to work over an arbitrary field $k$
of characteristic $0$. Fix $d\geq 2$
and let $(S_{d+1})_{\Delta}$ be the affine open subset 
in $S_{d+1}$ of forms defining smooth hypersurfaces in $\PP^{n-1}$.
An element $F\in (S_{d+1})_{\Delta}$ defines an isolated
homogeneous hypersurface singularity $F(x_1,\dots,x_n)=0$ in $k^n$. 
The Jacobian ideal $J_F:=(\partial F/\partial x_1,\dots, \partial F/\partial x_n)$
is a balanced complete intersection ideal, and so the Milnor
algebra $M_F:=S/J_F$ has a Macaulay inverse system given by the
associated form
\[
A(F):=\AA(\partial F/\partial x_1,\dots, \partial F/\partial x_n)\in D_{n(d-1)}.
\]
The morphism
$A\colon (S_{d+1})_{\Delta}\to D_{n(d-1)}$ gives rise to
an $\SL(n)$-contravariant
\[
S_{d+1} \to D_{n(d-1)}
\]
(see \cite{alper-isaev-kruzhilin} and \cite{isaev-contravariant} for details).

We will say 
that for $F,G\in (S_{d+1})_\Delta$, two singularities $F=0$ and $G=0$ are isomorphic if and only if 
\begin{equation}\label{E:hom}
\bar{k}[[x_1,\dots, x_n]]/(F)\simeq \bar{k}[[x_1,\dots,x_n]]/(G)
\end{equation}
as algebras over the algebraic closure $\bar k$ of $k$. This condition is equivalent to the existence of 
a matrix $C\in \GL(n)$, defined over $\bar k$, such that $G=C\cdot F$. Indeed, the isomorphism in \eqref{E:hom} lifts to an automorphism of the power series ring $\bar{k}[[x_1,\dots, x_n]]$ (see \cite[Lemma 1.23]{greuel-lossen-shustin}), which is given by a change of variables, and we take $C$ to be its linear part. Note, however, that such a $C$ 
does not have to exist over $k$ as the example of $F=x_1^4-x_2^4$ and $G=x_1^4+x_2^4$ in $\RR[x_1,x_2]$
illustrates. Nevertheless, Equation \eqref{E:hom} is equivalent to the equality of schemes $\GL(n)\cdot F=\GL(n)\cdot G$.

Our results imply that the morphism $A$ sends forms with non-zero
discriminant to polystable forms,
and from this fact we deduce an invariant-theoretic version of the Mather-Yau
theorem (see \cite{mather-yau}).

\begin{theorem}\label{T:invariant-MY} 
There exists a finite collection
of homogeneous $\SL(n)$-inva\-riants $\mathfrak{I}_1,\dots, \mathfrak{I}_N$ 
on $D_{n(d-1)}$ of equal degrees, defined over $k$,  
such that for any two forms $F, G\in (S_{d+1})_{\Delta}$, the isolated
homogeneous hypersurface singularities $F=0$ and $G=0$ are isomorphic 
if and only if
\[
[\mathfrak{I}_1(A(F)): \cdots:\mathfrak{I}_N(A(F))]
=[\mathfrak{I}_1(A(G)): \cdots:\mathfrak{I}_N(A(G))].
\]
\end{theorem}

\begin{remark} Our results show that the Mather-Yau theorem in the homogeneous situation can be extended to the case of an arbitrary field $k$ of characteristic 0 by stating that for $F, G\in (S_{d+1})_{\Delta}$, the singularities $F=0$ and $G=0$ are isomorphic if and only if 
$M_F\otimes_{k} \bar k$ and $M_G\otimes_{k} \bar k$ are isomorphic as $\bar k$-algebras.
The main novelty of Theorem \ref{T:invariant-MY} is in showing that one can check whether 
such an isomorphism exists simply by evaluating \emph{finitely many} $\SL(n)$-invariants on 
the associated forms of $M_F$ and $M_G$, 
and that this can be done without passing to the algebraic closure of $k$. 
\end{remark}   

To prove this result, we will need the following immediate consequence of
Theorem \ref{T:polystability} for the geometry of the associated form morphism
(cf. \S\ref{S:ci}): \[\AA\colon \Grass(n, S_d)_{\Res} \to
\PP D_{n(d-1)}.\]
 
\begin{corollary}\label{C:polystability}
 The induced morphism of GIT quotients
\[
\AA\gitq \SL(n) \colon \Grass(n, S_d)_{\Res}\gitq \SL(n) \to \PP
D_{n(d-1)}^{ss}\gitq \SL(n)
\]
is a locally closed immersion.
\end{corollary}
\begin{proof}
By \cite[\S2.5]{alper-isaev-assoc-binary}, $\Grass(n, S_d)_{\Res}$
maps isomorphically via $\AA$ to an $\SL(n)$-invariant open subset, say $O$,
of an $\SL(n)$-invariant closed subscheme $Z\hookrightarrow \PP D_{n(d-1)}$.
By \cite[Theorem 1.2]{fedorchuk-ss}, the image of $\AA$ lies
in the semistable locus $Z^{ss}$, and Theorem
\ref{T:polystability} implies that $O$ is a saturated open subset of $Z^{ss}$. 
The corollary now follows.
\end{proof}

\begin{proof}[{Proof of Theorem \ref{T:invariant-MY}}]
Note that the two singularities are isomorphic if and only if
$\GL(n)\cdot F=\GL(n)\cdot G$ in $S_{d+1}$, which by the GIT stability
of smooth hypersurfaces is equivalent
to the fact that $F$ and $G$ map to the same point
in the GIT quotient $\PP(S_{d+1})_{\Delta}\gitq \SL(n)$.

Recall from \cite[\S2.3]{alper-isaev-assoc-binary} that
the (projectivized) morphism $A\colon \PP(S_{d+1})_{\Delta} \to \PP
D_{n(d-1)}$ factors as the composition
of the gradient morphism \[\nabla\colon \PP(S_{d+1})_{\Delta} \to
\Grass(n, S_d)_{\Res},\]
defined by $\nabla F:= \langle \partial F/\partial x_1,\dots, \partial
F/\partial x_n\rangle$,
and the associated form morphism \[\AA\colon \Grass(n, S_d)_{\Res} \to
\PP D_{n(d-1)}.\]
By Corollary \ref{C:polystability},
the induced morphism of GIT quotients
\[
\AA\gitq \SL(n) \colon \Grass(n, S_d)_{\Res}\gitq \SL(n) \to \PP
D_{n(d-1)}^{ss}\gitq \SL(n)
\]
is a locally closed immersion.

Next, by \cite[Theorem 1.1]{fedorchuk-ss}, we have that $\nabla(F)$ is
polystable for every
$F\in \PP (S_{d+1})_{\Delta}$. Moreover, by \cite[Proposition 2.1 (2)]{fedorchuk-ss} 
the induced morphism on the
GIT quotients
\[
\nabla\gitq \SL(n) \colon \PP(S_{d+1})_{\Delta}\gitq \SL(n) \to
\Grass(n, S_d)_{\Res}\gitq \SL(n)
\]
is injective. We conclude that the composition morphism
\[
A\gitq \SL(n) \colon \PP(S_{d+1})_{\Delta}\gitq \SL(n) \to \PP
D_{n(d-1)}^{ss}\gitq \SL(n)
\]
is injective.
The theorem
now follows from the definition of the GIT quotient $\PP
D_{n(d-1)}^{ss}\gitq \SL(n)$
and the fact that the ring of $\SL(n)$-invariant forms on $D_{n(d-1)}$
is finitely generated
by the Gordan-Hilbert theorem.
\end{proof}

\begin{example}
We conclude with an example showing that our GIT stability results are
optimal as far as complete intersection
Artinian algebras are concerned. Consider the following quasi-homogeneous form:
\[
F(x_1,\dots,x_n)=x_1^{d_1+1}+\cdots+x_n^{d_n+1}.\]
Then the Milnor algebra of $F$ is
\[
M_F=S/(x_1^{d_1},\dots, x_n^{d_n}),
\]
which is a complete intersection Gorenstein Artin $k$-algebra of socle
degree $(d_1+\cdots+d_n)-n$. The homogeneous Macaulay inverse system
of this algebra (up to a non-zero scalar) is
\[
z_1^{d_1-1}\cdots z_n^{d_n-1}.
\]
Unless
$d_1=\cdots=d_n$,
this form is patently unstable with respect to the $\SL(n)$-action on
$D_{(d_1+\cdots+d_n)-n}$,
and so all homogeneous $\SL(n)$-invariants vanish on it.
\end{example}

\bibliographystyle{amsplain}

\end{document}